\documentclass{amsart}

\usepackage{intropapers}
\usepackage[margin=1in]{geometry}
\usepackage{cite}
\usepackage{hyperref}
\usepackage{url}
\usepackage{tikz-cd}

\title{Mod~$p$ Bernstein centres of $p$-adic groups}
\author{Andrea Dotto}
\date{}

\begin{document}
\begin{abstract}
We prove that the centre of the category of smooth mod~$p$ representations with fixed central character of a split semisimple $p$-adic group is a local ring.
\end{abstract}

\maketitle

\section{Introduction.}
Let~$F/\bQ_p$ be a finite extension with ring of integers~$\mO$ and residue field~$k$, and fix a uniformizer~$\pi_F$ of~$\mO$.
Let~$G/\mO$ be a split connected semisimple group, and fix a Borel subgroup $B/\mO$ with unipotent radical~$U$, and a maximal torus~$T/\mO$ contained in~$B$.
Write~$Z$ for the centre of~$G$.
We will sometimes use the same symbol to denote an algebraic group over~$F$ and its group of $F$-points, and we will write~$K = G(\mO)$.
Fix an algebraic closure $\cbF_p$ of~$\bF_p$, an embedding $k \to \cbF_p$, and a smooth character $\zeta: Z(F) \to \cbF_p^\times$.
We prove the following theorem.

\begin{thm}\label{maintheorem}
Assume that the Dynkin diagram of~$G$ is connected.
Let~$\mZ$ be the centre of the category of smooth $\cbF_p[G]$-representations with central character~$\zeta$.
Then
\benum
\item $\mZ$ is a local $\cbF_p$-algebra with residue field~$\cbF_p$, and
\item the maximal ideal of~$\mZ$ acts by locally nilpotent operators on all smooth~$\cbF_p[G]$-representations.
\eenum
\end{thm}

Here we say an endomorphism $N: \pi \to \pi$ of an $\cbF_p[G]$-representation~$\pi$ is locally nilpotent if for all~$v \in \pi$ there exists a positive integer~$n$ such that $N^n v = 0$, with~$n$ possibly depending on~$v$.

\subsection{The work of Ardakov and Schneider.}
Theorem~\ref{maintheorem} was proved in~2020 in the context of forthcoming work with Emerton and Gee on localization theory for $p$-adic representations of~$\GL_2(\bQ_p)$, and was intended to appear together with that work. 
Since then, Ardakov and Schneider have used completely different techniques to prove a stronger result, namely~\cite[Theorem~6.14]{AScentres}, which implies that the maximal ideal of the local ring~$\mZ$ is in fact zero (at least when the group is adjoint). 
We have decided to write this note in the hope that our methods may still be of independent interest.

\subsection{Remarks on our assumptions.}
Since stronger results are available in~\cite{AScentres}, and in order to streamline the main argument and gain direct access to results in the literature, we have added a few simplifying assumptions to the statement of the theorem (namely that the group is split semisimple and the Dynkin diagram is connected). Because of these assumptions our theorem does not apply directly to the group~$\GL_2(\bQ_p)$. To remedy this, we now give a proof of theorem~\ref{maintheorem} for~$G = \GL_2(\bQ_p)$: the argument in the general case proceeds along similar lines, although there are some complications arising from the larger size of the Hecke algebras.

\subsection{The case of~$\GL_2(\bQ_p)$.}
Let~$F = \bQ_p, k = \bF_p$ and~$G = \GL_2$.
Let~$T \in \mZ$ and let~$\sigma$ be an irreducible~$\cbF_p[K]$-representation of central character~$\zeta |_K$.
Let~$Z$ act on~$\sigma$ by the character~$\zeta$, and write
\[
\mH_G(\sigma) = \End_G(\cInd_{KZ}^G \sigma)
\] 
for the Hecke algebra of~$\sigma$.
The central element~$T$ acts on~$\cInd_K^G(\sigma)$ by an element of~$\mH_G(\sigma)$ that we denote~$T_\sigma$.

The main point is to prove that~$T_\sigma$ is a scalar~$\lambda_\sigma \in \cbF_p$, and that~$\lambda_\sigma$ is independent of~$\sigma$.
Once this is done, one deduces theorem~\ref{maintheorem} by the same argument as in the general case, which is in section~\ref{endofproof}.
Assuming that~$T_\sigma$ is scalar, one deduces that~$\lambda_{\sigma_1} = \lambda_{\sigma_2}$ for nonprojective Serre weights~$\sigma_1$ and~$\sigma_2$ by reducing to the case that~$\Ext^1_{KZ}(\sigma_1, \sigma_2) \ne 0$ and using lemma~\ref{nonsplit}. 
The reduction step uses the structure of the blocks of~$\cbF_p[\GL_2(\bF_p)]$ to connect any two nonprojective representations~$\sigma_1, \sigma_2$ with the same central character by a chain of extensions.
(A more general and precise version of this statement, which follows from a result of Humphreys, is given in lemma~\ref{connectextensions}. It is related to the fact that the moduli space of the Emerton--Gee stack for~$\GL_2(\bQ_p)$ is connected.)

Now we fix a nonprojective Serre weight $\sigma$ and we prove that~$T_\sigma$ is a scalar.
For this it suffices to prove that~$\supp(T_\sigma) \subset KZ$, because the Hecke operators supported on~$KZ$ are of the form $\cInd_{KZ}^G(\alpha)$ for some $KZ$-linear map $\alpha: \sigma \to \sigma$, and these are all scalar.
To do so, we take a projective envelope $R_\sigma \to \sigma$ in the category of~$\cbF_p[\GL_2(\bF_p)]$-representations. 
Since the centre of~$\GL_2(\bF_p)$ has order coprime to~$p$, the module $R_\sigma$ has the same central character as~$\sigma$, and we can study the action of~$T$ on the exact sequence
\[
0 \to \cInd_{KZ}^G(\rad\, R_\sigma) \to \cInd_{KZ}^G(R_\sigma) \to \cInd_{KZ}^G(\sigma) \to 0.
\]
Write~$T_\sigma^+$ for the image of~$T$ in~$\mH_G(R_\sigma)$.
If~$T_\sigma$ is not supported in~$KZ$, then neither is~$T_{\sigma}^+$.

Now we use the fact that elements of~$\mH_G(R_\sigma)$ supported in double cosets other than~$KZ$ are in bijection with $T(k)$-linear maps
\[
(R_{\sigma})_{\lbar U(k)} \to R_{\sigma}^{U(k)}
\]
where~$U$, resp. $\lbar U$ is the upper-triangular, resp. lower-triangular, unipotent subgroup of~$G$.
Furthermore, the map corresponding to~$T_\sigma^+$ (and still denoted~$T_\sigma^+$) has the additional property that
\[
\begin{tikzcd}
(R_\sigma)_{\lbar U(k)} \arrow[r, "T_\sigma^+"] \arrow[d] & R_\sigma^{U(k)} \arrow[d]\\
\sigma_{\lbar U(k)} \arrow[r, "T_\sigma"] & \sigma^{U(k)}
\end{tikzcd}
\] 
commutes.
Hence $R_\sigma \to \sigma$ yields a surjection
\[
R_\sigma^{U(k)} \to \sigma^{U(k)}.
\]
By Frobenius reciprocity, we deduce that there is a surjection $\Ind_B^G(\sigma^{U(k)}) \to R_\sigma$, hence~$R_\sigma$ is a direct summand of the finite parabolic induction $\Ind_B^G(\sigma^{U(k)})$.
This never happens if~$\sigma$ is not projective, and we deduce that~$T_\sigma$ is supported in~$KZ$.

To conclude, we need to treat the case when~$\sigma$ is a projective irreducible~$\cbF_p[\GL_2(\bF_p)]$-representation, or equivalently a twist of the Steinberg representation of~$\cbF_p[\GL_2(\bF_p)]$. 
We need to prove that if~$\chi: \GL_2(\bF_p) \to \cbF_p^\times$ is a character then~$T$ acts by a scalar~$\lambda_{\chi\otimes\St}$ on~$\cInd_{KZ}^G(\chi\otimes\St)$, and that~$\lambda_{\chi\otimes\St} = \lambda_\chi$.
This is immediate from the existence of an injective $G$-linear map
\[
\cInd_{KZ}^G(\chi\otimes \St) \to \cInd_{KZ}^G(\chi).
\]

\subsection{Acknowledgments.}
We thank Matthew Emerton and Toby Gee for helpful conversations on these and related matters. The author was supported by the James D. Wolfensohn Fund at the Institute for Advanced Study during the writing of this note.

\section{Preliminaries.}

\subsection{Hecke algebras.}
In this paragraph we let~$G$ be a locally profinite group and~$H$ an open subgroup of~$G$.
We fix a finite-dimensional, irreducible and smooth $\lbar \bF_p[H]$-representation~$\sigma$ and a generator~$v$ of~$\sigma$ over~$\cbF_p[H]$.

The compact induction $\cInd_H^G(\sigma)$ is a smooth cyclic $\cbF_p[G]$-module, generated by the function~$[1, v]$ supported in~$H$ and sending the identity to~$v$.
The Hecke algebra $\End_G(\cInd_{H}^G(\sigma))$ is isomorphic to the algebra $\mH_G(\sigma)$ of compactly supported functions
\[
\varphi: G \to \End_{\cbF_p}(V_\sigma) \text{ satisfying } \varphi(h_1 t h_2) = h_1 \circ \varphi(t) \circ h_2
\] 
under convolution.
If~$\varphi$ is such a function then we define the support of~$\varphi$ as the set of double cosets~$HtH$ such that~$\varphi|_{HtH} \ne 0$.
This set will also be referred to as the support of the associated endomorphism~$T_\varphi \in \End_G(\cInd_{H}^G(\sigma))$.

%

We will also need the following description of the Hecke algebra.
For any system~$\mH$ of representatives in~$G$ of~$H \backslash G / H$ and any $t \in \mH$, define~$\sigma(t)$ as the space of functions in~$\cInd_H^G(\sigma)$ supported on~$HtH$.
It is an $H$-representation, and by Frobenius reciprocity and the fact that~$\sigma$ is finitely generated, $\mH_G(\sigma)$ is also $\cbF_p$-linearly isomorphic to
\begin{equation}\label{Heckesum}
\bigoplus_{t \in \mH}\Hom_H \left ( \sigma, \sigma(t) \right )
\end{equation}
We can reformulate this by defining a functor $\ad(t)^*: \Rep_{\cbF_p}(H \cap tHt^{-1}) \to \Rep_{\cbF_p}(H \cap t^{-1}Ht)$ as follows: it sends an object~$(\tau, V)$ to the representation of~$H \cap t^{-1} H t$ on the same vector space~$V$ but the action specified by~$x \mapsto \tau(\ad(t)x) = \tau(txt^{-1})$.
It is the identity on morphisms.
Then there is an $H$-linear isomorphism
\[
\sigma(t) \to \Ind_{H \cap t^{-1}H t}^H \left ( \ad(t)^* \sigma \right )
\]
sending a function~$f$ to the function~$h \mapsto f(th)$.

\subsection{Cartan decomposition.}
Write~$\Phi = \Phi(G, T)$ for the root system of~$T$ acting on~$\fg = \Lie(G)$. 
We work with the positive roots~$\Phi^+$ determined by~$B$.
Let~$\Delta$ be the corresponding system of simple roots.
The monoid of antidominant coweights is
\[
X_*(T)_- = \{\lambda \in X_*(T) : (\lambda, \alpha) \leq 0 \text{ for all } \alpha \in \Phi^+\}.
\]
Evaluating at the uniformizer~$\pi_F$ we obtain a bijection of~$X_*(T)_-$ onto $T^- / T(\mO)$, where
\[
T^- = \{t \in T(F) : v_F(\alpha(t)) \leq 0 \text{ for all } \alpha \in \Phi^+\}.
\]
The refined Cartan decomposition is
\[
G = \coprod_{\lambda \in X_*(T)_-} K \lambda(\pi_F) K.
\]

\subsection{Parabolic subgroups.}
If~$\lambda \in X_*(T)$ is a cocharacter, we write~$P_\lambda$ for the associated parabolic subgroup of~$G$ and~$U_\lambda$, resp. $L_\lambda$ for its unipotent radical, resp. Levi factor.
If~$\lambda \in X_*(T)_-$ and~$t = \lambda(\pi_F)$ we can identify the image of $K \cap t^{-1}Kt$ in~$G(k)$ with the group of $k$-points of~$P_\lambda$.
Namely, if~$\red: G(\mO) \to G(k)$ is the reduction map, then
\[
\red(K \cap t^{-1}K t) = P_\lambda(k).
\]
This is proved in~\cite[Proposition~3.8]{HerzigSatake}.

\subsection{Serre weights.}
As usual, by a Serre weight of~$G$ we mean an irreducible representations of $\cbF_p[G(k)]$.
By~\cite[Lemma~2.3]{Herzigreps}, and references therein, if~$\sigma$ is a Serre weight and~$U_\lambda$ is the unipotent radical of a parabolic subgroup of~$G$ then~$\sigma^{U_\lambda(k)}$ is an irreducible $L_\lambda(k)$-representation.
Whenever~$\sigma$ is a Serre weight for~$G$, we will write~$R_\sigma \to \sigma$ for a projective envelope in the category of~$\cbF_p[G(k)]$-representations.

\subsection{Satake isomorphism.}
Let~$\sigma$ be a Serre weight for~$G$.
The centre~$Z$ of~$G$ is a finite flat group scheme of multiplicative type over~$\mO$, and so~$Z(F) = Z(\mO) \subset K$ and the $p$-coprime part of $Z(\mO)$ identifies under reduction mod~$p$ with~$Z(k)$. 
It follows that the $\cbF_p^\times$-valued characters of~$Z(F)$ and~$Z(k)$ are identified.
We will assume that the central character of~$\sigma$ equals our fixed character $\zeta: Z(F) \to \cbF_p^\times$.
The Hecke algebra~$\End_{K}^G(\cInd_K^G(\sigma))$ will be denoted~$\mH_G(\sigma)$.

The Satake transform
\[
\mS_G : \mH_G(\sigma) \to \mH_T(\sigma^{U(k)})
\]
is an injective~$\cbF_p$-algebra homomorphism defined in~\cite{HerzigSatake}.
Its image consists of the functions that are supported on~$T^-$.
We will need some information about its effect on the support of functions.
For this, we use that~$K \backslash G / K$ is in bijection with~$X_*(T)_-$ (via the Cartan decomposition) and~$T(\mO) \backslash T / T(\mO)$ is in bijection with~$X_*(T)$.
The group~$X_*(T)$ has a partial order~$\geq_\bR$, where~$\lambda \geq_\bR \mu$ if~$\lambda - \mu$ is a nonnegative real linear combination of the positive coroots of~$(G, B, T)$.
For each~$\lambda \in X_*(T)_-$ the space of elements of~$\mH_G(\sigma)$ supported in~$K \lambda(\pi_F) K$ is one-dimensional, and a certain generator~$T_\lambda$ is singled out in~\cite[Section~1.4]{HerzigSatake}.
Similarly, if~$\mu \in X_*(T)$ we write~$\tau_\mu$ for the generator of the space of elements of~$\mH_T(\sigma^{U})$ supported in~$T(\mO) \mu(\pi_F) T(\mO)$ such that~$\tau_\mu(\mu(\pi_F)) = 1$.

\begin{pp}\label{supports}
Let~$X \in \mH_G(\sigma)$.
Then the minimal nonzero elements of the supports of~$X$ and~$\mS_G(X)$ coincide.
\end{pp}
\begin{proof}
This follows from the discussion after~\cite[Proposition~1.4]{HerzigSatake}, where it is asserted that we have
\[
\tau_\lambda = \sum_{\substack{\mu \in X_*(T)_- \\  \mu \geq_\bR \lambda}} d_\lambda(\mu) \mS_G(T_\mu)
\]
for certain scalars~$d_\lambda(\mu) \in \cbF_p$ with~$d_\lambda(\lambda) = 1$.
Writing~$\mS_G(X)$ in the basis given by the~$\tau_\mu$ we obtain an expression $\mS_G(X) = \sum_{\lambda \in X_*(T)_-} \beta_\lambda \tau_\lambda$, and this implies the claim upon applying~$\mS_G^{-1}$ to both sides.
\end{proof}

\subsection{Change of weight.}
Assume that~$\sigma_1, \sigma_2$ are Serre weights for~$G$ and that~$\sigma_1^{U(k)} \cong \sigma_2^{U(k)}$. 
The space
\[
\Hom_{T(k)}(\sigma_1^{U(k)}, \sigma_2^{U(k)})
\]
is one-dimensional over~$\cbF_p$, and every nonzero element induces the same isomorphism
\[
\iota: \mH_T(\sigma_1^{U(k)}) \isom \mH_T(\sigma_2^{U(k)}).
\]
This map~$\iota$ preserves the support of functions.
The Satake transform then yields an isomorphism
\[
\mH_G(\sigma_1) \isom \mH_G(\sigma_2).
\]
that we still denote~$\iota$.
By proposition~\ref{supports}, it preserves the minimal nonzero elements of the support of functions.
By~\cite[Proposition~6.2]{Herzigreps}, the space
\[
\Hom_G \left ( \cInd_{K}^G(\sigma_1), \cInd_{K}^G(\sigma_2) \right )
\]
is not zero and all of its elements are equivariant for~$\iota$.
Furthermore, by the proof~\cite[Corollary~6.5]{Herzigreps} we know that the nonzero elements of this space are all injective.

\section{Proof of the main theorem.}
We are going to prove theorem~\ref{maintheorem} by implementing a similar strategy as for~$\GL_2(\bQ_p)$.
Let~$T \in \mZ$, let~$\sigma$ be a Serre weight for~$G$ with central character~$\zeta$, and write~$T_\sigma$ for the image of~$T$ in $\mH_G(\sigma) = \End_H(\cInd_K^G(\sigma))$.
We begin by proving that~$T_\sigma = \lambda_\sigma$ for some~$\lambda_\sigma \in \cbF_p$.
This is theorem~\ref{scalars}, which requires some preliminary steps.
Recall that $R_\sigma \to \sigma$ is a projective envelope of~$\sigma$ in the category of~$\cbF_p[G(k)]$-representations, and that it has the same central character as~$\sigma$ (since the centre of~$G(k)$ has order coprime to~$p$). 

\begin{pp}\label{exactsequence}
Write
\[
T_\sigma = \sum_{\lambda \in X_*(T)_-} a_\lambda T_\lambda.
\]
Assume that~$a_\mu \ne 0$ for some~$\mu \ne 0$, or equivalently that~$K\mu(\pi_F)K \subset \supp T_\sigma$.
Then there is an exact sequence
\[
0 \to (\rad R_\sigma)^{U_{-\mu}(k)} \to R_\sigma^{U_{-\mu}(k)} \to \sigma^{U_{-\mu}(k)} \to 0.
\]
\end{pp}
\begin{proof}
Since~$\sigma^{U_{-\mu}(k)}$ is an irreducible representation of~$L_{-\mu}(k)$, it suffices to prove that the map $R_\sigma^{U_{-\mu}(k)} \to \sigma$ is not zero.
Consider the image~$T_\sigma^+$ of~$T$ in~$\mH_G(R_\sigma)$.
Since~$T$ is central, there is a commutative diagram
\[
\begin{tikzcd}
\cInd_{K}^G(R_\sigma) \arrow[r, "T_\sigma^+"] \arrow[d] & \cInd_{K}^G(R_\sigma) \arrow[d, "\cInd_{K}^G(\pr)"]\\
\cInd_{K}^G(\sigma) \arrow[r, "T_\sigma"] & \cInd_{K}^G(\sigma)
\end{tikzcd}
\]
inducing another
\[
\begin{tikzcd}
R_\sigma \arrow[r, "T_\sigma^+"] \arrow[d] & \bigoplus_{\lambda \in X_*(T)_-} R_\sigma(\lambda) \arrow[d]\\
\sigma \arrow[r, "T_\sigma"] & \bigoplus_{\lambda \in X_*(T)_-} \sigma(\lambda)
\end{tikzcd}
\]
by Frobenius reciprocity.
Since~$a_\mu \ne 0$, we know that $T_\sigma(v_\sigma)$ has nonzero projection to the summand~$\sigma(\mu)$. 
This implies that~$T_\sigma^+(v_\sigma^+)$ has nonzero projection to~$R_\sigma(\mu)$.
Hence the composition of $K$-linear maps
\[
R_\sigma \xrightarrow{T^+_\sigma} R_\sigma(\mu) \to \sigma(\mu)
\]
is nonzero. 

Recall that~$R_\sigma(\mu)$ is isomorphic to $\Ind_{K \cap t^{-1}Kt}^{K} \ad(t)^*R_\sigma$, where $t = \mu(\pi_F)$ and $\ad(t)^* R_\sigma$ is the representation of~$t^{-1}K t$ on the same space as~$R_\sigma$ defined by $x \mapsto R_\sigma(txt^{-1})$.
Again by Frobenius reciprocity, we obtain a map
\[
R_\sigma \xrightarrow{T^+_\sigma} R_\sigma
\]
which is equivariant for $\ad(t): K \cap t^{-1}K t \to K \cap t K t^{-1}$, and whose composition with the projection $R_\sigma \to \sigma$ is nonzero.

Then it suffices to prove that the image of~$T^+_\sigma: R_\sigma \to R_\sigma$ is contained in~$R_{\sigma}^{U_{-\mu}(k)}$.
Let~$x \in R_\sigma$ and~$g \in U_{-\mu}(k)$. 
Since the projection of~$g$ to the Levi quotient of~$P_{-\mu}(k)$ is the identity, \cite[Proposition~3.8]{HerzigSatake} implies that we can find~$\tld g \in K \cap t^{-1}K t$ such that
\[
\red(\tld g, t\tld gt^{-1}) = (\id, g).
\]
By the equivariance properties of~$T^+_\sigma$, we deduce
\[
T^+_\sigma(x) = T^+_\sigma(\tld g \cdot x) = t \tld g t^{-1}T^+_\sigma(x) =  g T^+_\sigma(x)
\]
and the claim follows.
\end{proof}

\begin{pp}\label{samesupports}
Let~$\sigma_1, \sigma_2$ be Serre weights for~$G$ with~$\sigma_1^{U(k)} \cong \sigma_2^{U(k)}$.
Let~$\lambda \in X_*(T)_-$ be a minimal nonzero element of the support of~$T_{\sigma_2}$ under the partial order~$\geq_\bR$.
Then~$K\lambda(\pi_F)K \subset \supp(T_{\sigma_1})$.
\end{pp}
\begin{proof}
By our discussion of change of weight there exists an injective $\cbF_p[G]$-linear map
\[
\tau: \cInd_K^G(\sigma_1) \to \cInd_K^G(\sigma_2)
\]
that is furthermore equivariant for the isomorphism $\iota: \mH_G(\sigma_1) \isom \mH_G(\sigma_2)$ induced by any $\cbF_p[T(k)]$-linear isomorphism $\sigma_1^{U(k)} \isom \sigma_2^{U(k)}$.
In addition, the minimal nonzero elements of the supports of~$T_{\sigma_2}$ and~$\iota^{-1}(T_{\sigma_2})$ coincide.
Since~$\tau$ is $\cbF_p[G]$-linear and~$T \in \mZ$, we have
\[
\tau \circ T_{\sigma_1} = T_{\sigma_2} \circ \tau,
\]
and since~$\tau$ is $\iota$-equivariant, we have
\[
\tau \circ \iota^{-1}T_{\sigma_2} = T_{\sigma_2}\circ \tau.
\]
Since~$\tau$ is injective, this implies~$T_{\sigma_1} = \iota^{-1}(T_{\sigma_2})$, and the proposition follows.
\end{proof}

\begin{corollary}\label{exactsequences}
Assume~$K \lambda(\pi_F) K$ is a minimal nonzero element of~$\supp(T_\sigma)$.
If~$\rho$ is a Serre weight with~$\sigma^{U(k)} \cong \rho^{U(k)}$, then there is an exact sequence
\[
0 \to (\rad R_\rho)^{U_{-\lambda}(k)} \to R_\rho^{U_{-\lambda}(k)} \to \rho^{U_{-\lambda}(k)} \to 0.
\]
\end{corollary}
\begin{proof}
Immediate from propositions~\ref{exactsequence} and~\ref{samesupports}.
\end{proof}

\begin{thm}\label{scalars}
There exists~$\lambda_\sigma \in \cbF_p$ such that~$T_\sigma = \lambda_\sigma \in \mH_G(\sigma)$.
\end{thm}
\begin{proof}
Since the Hecke operators with support equal to~$K$ are precisely those of the form~$\cInd_K^G(\alpha)$ for a $K$-linear map $\alpha: \sigma \to \sigma$, it suffices to prove that~$\supp(T_\sigma) = K$.
Assume for a contradiction that~$K \lambda(\pi_F) K$ is a minimal nonzero element of $\supp(T_\sigma)$.
Let~$P  = P_{-\lambda}$ and write~$L = L_{-\lambda}$ for its Levi quotient and~$V = U_{-\lambda}$ for its unipotent radical.
Let~$Q_\sigma$ be a projective envelope of~$\sigma^{V(k)}$ in the category of~$\cbF_p[L(k)]$-representations.
We are going to deduce that $Q_\sigma^+ = \Ind_{P(k)}^{G(k)}(Q_\sigma)$ is a projective~$\cbF_p[G(k)]$-module.
To see that this is impossible, we use that~$\Res^G_H$ preserves projectives whenever~$H \subset G$ are finite groups, since it is left adjoint to~$\Ind_H^G$ which is exact.
Since~$Q_\sigma$ is a direct summand of $\Res^{G(k)}_{P(k)}\Ind_{P(k)}^{G(k)}(Q_\sigma)$, it would be a projective~$\cbF_p[P(k)]$-module.
But this is not true, since the restriction of~$Q_\sigma$ to~$V(k)$ is a trivial representation and~$V(k)$ is a $p$-group.

To prove that~$Q_\sigma^+$ is projective, we first compute its cosocle.
For any irreducible $\cbF_p[G(k)]$-representation~$\rho$ we have
\[
\Hom_{G(k)}(Q_\sigma^+, \rho) \cong \Hom_{P(k)}(Q_\sigma, \rho) \cong \Hom_{L(k)}(Q_\sigma, \rho^{V(k)})
\]
which is one-dimensional if~$\sigma^{V(k)} \cong \rho^{V(k)}$, and zero otherwise.

Assume~$\rho^{V(k)} \cong \sigma^{V(k)}$.
Since~$\lambda$ is antidominant, the parabolic~$P_{-\lambda}$ contains the Borel subgroup~$B$ and so $V(k) \subset U(k)$, which implies that $\sigma^{U(k)} \cong \rho^{U(k)}$.
By corollary~\ref{exactsequences}, we have exact sequences
\[
0 \to (\rad R_\rho)^{V(k)} \to R_\rho^{V(k)} \to \rho^{V(k)} \to 0
\]
and so we have a $P(k)$-linear map
\[
Q_\sigma \to R_\rho^{V(k)} \subset R_\rho
\]
whose composition with~$R_\rho \to \rho$ is not zero.
By Frobenius reciprocity, we obtain a $G(k)$-linear surjection
\[
Q_\sigma^+ \to R_\rho.
\]
Since~$R_\sigma$ is projective, this implies that we can write~$Q_\sigma^+ = M_\rho \oplus R_\rho$ for some $\cbF_p[G(k)]$-module~$M_\rho$.

By the same argument, if~$\rho'$ is another Serre weight with~$(\rho')^{V(k)} \cong \sigma^{V(k)}$ then $Q_\sigma^+$ surjects onto~$R_{\rho'}$.
This implies that~$M_\rho$ also surjects onto~$R_{\rho'}$, because~$R_{\rho'}$ has a unique maximal submodule, and its cosocle is not~$\rho$, and so if~$M_\rho \to R_{\rho'}$ has image in~$\rad(R_{\rho'})$ then so does~$Q_\sigma^+ \to R_{\rho'}$.
Hence we can write
\[
Q_{\sigma}^+ = M_{\rho, \rho'} \oplus R_{\rho} \oplus R_{\rho'}
\]
for some~$M_{\rho, \rho'}$.
Repeating this argument, we find that
\[
Q_{\sigma}^+ = \bigoplus_{\rho: \rho^{V(k)} \cong \sigma^{V(k)}} R_\rho,
\]
since the right-hand side is a summand of the left-hand side and they have the same cosocle.
Hence~$Q_{\sigma}^+$ is projective.
\end{proof}

\subsection{Independence of~$\sigma$.}\label{independencesigma}
Now we prove that the scalar~$\lambda_\sigma$ is independent of~$\sigma$.
If~$\sigma_1^{U(k)} \cong \sigma_2^{U(k)}$, the existence of an injective $\cbF_p[G]$-linear map
\[
\cInd_K^G(\sigma_1) \to \cInd_K^G(\sigma_2)
\]
immediately implies that~$\lambda_{\sigma_1} = \lambda_{\sigma_2}$. 
To go further, we make use of a result about the blocks of~$\cbF_p[G(k)]$.

By the main theorem of~\cite{Humphreysdefect} the defect group of a block of~$\cbF_p[G(k)]$ is either the trivial subgroup of~$G(k)$ or the Sylow $p$-subgroup~$U(k)$ of~$G(k)$.
A block has trivial defect group if and only if it contains a projective simple module, which is then uniquely determined by the block, and the projective simple modules for~$\cbF_p[G(k)]$ are precisely the twists of the Steinberg representation.
By~\cite[Section~5(b)]{Humphreysdefect}, the number of blocks whose defect group is~$U$ is equal to the order of the centre of $G(k)$.
So the blocks of highest defect are in bijection with the central characters of~$G(k)$, and the bijection is specified by the action of the centre on any simple module in the block.
We deduce the following lemma.
\begin{lemma}\label{connectextensions}
If~$\sigma, \sigma'$ are irreducible, non-projective representations of~$\cbF_p[G(k)]$ with the same central character, then there exists a sequence
\[
\sigma_0 = \sigma, \sigma_1, \ldots, \sigma_n = \sigma'
\]
of simple~$\cbF_p[G(k)]$-modules such that~$\Ext^1_{G(k)}(\sigma_i, \sigma_{i+1}) \ne 0$ or~$\Ext^1(\sigma_{i+1}, \sigma_i) \ne 0$ for all~$i$.
\end{lemma}

Since every Steinberg twist has the same $U(k)$-invariants as some nonprojective representation (for instance, a character) it suffices by change of weight to prove that~$\lambda_{\sigma} = \lambda_{\sigma'}$ provided that there exists a nonsplit $\cbF_p[G(k)]$-extension
\[
0 \to \sigma \to X \to \sigma' \to 0.
\]
Assume~$\lambda_\sigma \ne \lambda_{\sigma'}$ and let~$T-\lambda_{\sigma} \in \mZ$ act on the short exact sequence
\[
0 \to \cInd_{K}^G(\sigma) \xrightarrow{\iota} \cInd_{K}^G(X) \xrightarrow{\pr} \cInd_{K}^G(\sigma') \to 0.
\]
We find that $T-\lambda_{\sigma}$ factors through a morphism
\[
\alpha: \cInd_{K}^G(\sigma') \to \cInd_{K}^G(X)
\]
whose composition with~$\pr$ is multiplication by the nonzero scalar~$\lambda_{\sigma'} - \lambda_\sigma$.
But then~$(\lambda_{\sigma'}-\lambda_\sigma)^{-1}\alpha$ is a section of~$\pr$, contradicting the following lemma.

\begin{lemma}\label{nonsplit}
Assume that
\begin{equation}\label{extension}
0 \to \sigma \to X \to \sigma' \to 0
\end{equation}
is a nonsplit $\cbF_p[G(k)]$-extension.
Then the extension
\begin{equation}\label{inductextension}
0 \to \cInd_{K}^G(\sigma) \xrightarrow{\iota} \cInd_{K}^G(X) \xrightarrow{\pr} \cInd_{K}^G(\sigma') \to 0
\end{equation}
is not split.
\end{lemma}
\begin{proof}
This is because, by the Mackey decomposition, the unit~$u$ of the Frobenius reciprocity adjunction is a split injection.
More precisely, assume that~(\ref{inductextension}) split and choose a $G$-linear retraction $r: \cInd_{K}^G(X) \to \cInd_K^G(\sigma)$ of~$\iota$.
There is a commutative diagram
\[
\begin{tikzcd}
\sigma \arrow{r}{u_\sigma} \arrow[d] & \cInd_K^G(\sigma) \arrow{d}{\iota}\\
X \arrow{r}{u_X} & \cInd_K^G(X).
\end{tikzcd}
\]
If~$r_\sigma$ is a $K$-linear retraction of~$u_\sigma$, then the composition $r_\sigma \circ r \circ u_X$ is a $K$-linear retraction of~$\sigma \to X$, contradicting the assumption that~(\ref{extension}) is not split.
\end{proof}

\subsection{End of proof.}\label{endofproof}
Since we know that the scalar~$\lambda_\sigma \in \cbF_p$ does not depend on~$\sigma$, we denote it by~$\lambda$.
We have an $\cbF_p$-linear ring homomorphism
\[
\alpha: \mZ \to \cbF_p, T \mapsto \lambda,
\]
and it suffices to prove that if~$T \in \ker(\alpha)$ and~$\pi$ is a smooth $\cbF_p[G]$-representation with central character~$\zeta$ then~$T$ is locally nilpotent on~$\pi$.
Indeed, it then follows that the geometric series for~$(1-T)^{-1}$ converges at each vector in~$\pi$ and defines an inverse for~$1-T$ in~$\mZ$, since it is locally a polynomial in~$T$. 
So $\ker(\alpha)$ is the only maximal ideal of~$\mZ$.

So let~$x \in \pi$.
Since~$\pi$ is smooth, the representation~$\tau = \langle K \cdot x \rangle$ generated by~$x$ under~$\cbF_p[K]$ is finite-dimensional, and there is a $G$-linear map
\[
\cInd_K^G(\tau) \to \pi
\]
whose image contains~$x$.
So it suffices to prove that~$T$ is nilpotent on~$\cInd_K^G(\tau)$.
But this follows from the assumption that~$T$ is zero on all Serre weights~$\sigma$, since~$\cInd_K^G(\tau)$ has a finite filtration whose subquotients are of the form~$\cInd_K^G(\sigma)$.

\bibliographystyle{amsalpha}
\bibliography{refpapers}

\end{document}